\documentclass{article}
\usepackage{amsfonts,amsmath,amsthm,amssymb}
\usepackage{graphics, epsfig}
\usepackage{color}
\usepackage{appendix}
\usepackage{ulem}
\usepackage[makeroom]{cancel}
\usepackage{fancyhdr}
\usepackage{centernot}
\usepackage{tikz-cd}
\usepackage{mathtools}
\usepackage{ stmaryrd }

 \usepackage[usenames,dvipsnames]{pstricks}
 \usepackage{pst-grad} 
 \usepackage{pst-plot} 
\allowdisplaybreaks

\let\TeXchi\chi
\newbox\chibox
\setbox0 \hbox{\mathsurround0pt $\TeXchi$}
\setbox\chibox \hbox{\raise\dp0 \box 0 }
\def\chi{\copy\chibox}

\newtheorem{proposition}{Proposition}[section]
\newtheorem{theorem}{Theorem}[section]
\newtheorem{definition}{Definition}[section]
\newtheorem{example}{Example}[section]

\newtheorem{corollary}{Corollary}[section]

\newtheorem{problem}{Problem}[section]

\numberwithin{equation}{section}
\numberwithin{theorem}{section}
\numberwithin{definition}{section}
\numberwithin{example}{section}
\numberwithin{proposition}{section}
\numberwithin{lemma}{section}
\numberwithin{remark}{section}
\newcommand\blfootnote[1]{%
  \begingroup
  \renewcommand\thefootnote{}\footnote{#1}%
  \addtocounter{footnote}{-1}%
  \endgroup
}
\begin{document}
\title{On the critical points of entire functions}
\author
{Manuel Norman}
\date{}
\maketitle
\begin{abstract}
\noindent Several years ago, Aziz and Zargar, while considering some questions related to Sendov's conjecture, solved a problem posed by Brown (see [1,2]), showing that any complex polynomial of degree $n$ with a single zero at $z=0$ does not have any critical point in $B(0,1/n)$. More recently, this result has been generalised in [3] by Zargar and Ahmad. The aim of our paper is to extend the result to some classes of complex entire functions. We will show that, under some conditions on the zeros $a_n$ of $f$, $f'$ has no roots in $B(0,t) \setminus \lbrace 0 \rbrace$ for a certain $t$ depending on the values of $|a_n|$.
\end{abstract}
\blfootnote{Author: \textbf{Manuel Norman}; email: manuel.norman02@gmail.com
}
\section{Introduction}
Since its formulation, Sendov's conjecture has attracted the attention of many mathematicians. The question is to prove that, given a complex polynomial $p$ with roots inside the unit disk, for each zero $z_k$ of $p$ there exists a critical point (i.e. a zero of $p'$) $\xi_j$ such that the distance $|z_k - \xi_j| \leq 1$. The fact that the critical points are inside the unit disk easily follows by Gauss-Lucas Theorem, so this conjecture is an extension of such important result. We refer to [4] for a proof of Sendov's conjecture for polynomials with degree $n<9$ and to [5] for the recent proof of Tao, which settled all the degrees large enough.\\
This problem also led to some other interesting questions related to the critical points of (complex) polynomials. In particular, Brown [2] noticed that the polynomial $z(z-1)^n$ has a critical point with absolute value $1/n$, and conjectured that any $p \in \mathbb{C}[z]$ of the form:
\begin{equation}\label{Eq:1.1}
p(z)=z \prod_{k=1}^{n-1} (z-z_k)
\end{equation}
with $|z_k| \geq 1$ for all $k$, has no critical point with absolute value $<1/n$.\\
Ten years later, Aziz and Zargar came out with a proof of this conjecture:
\begin{proposition}[Aziz-Zargar]\label{Prop:1.1}
Let $p \in \mathbb{C}[z]$ be a polynomial as in \eqref{Eq:1.1}, with $|z_k| \geq 1$ for all $k$. Then, $p$ has no critical point in the ball $B(0,1/n):=\lbrace z \in \mathbb{C}: |z|<1/n \rbrace$.
\end{proposition}
Actually, Aziz and Zargar also proved that polynomials with a multiple zero at $z=0$ (say, with multiplicity $m$) and degree $n$ have no critical points in $B(0,m/n) \setminus \lbrace 0 \rbrace$. This result has been extended to the $j$-th derivative of a polynomial in [3], where Zargar and Ahmad obtained the following generalisation:
\begin{proposition}[Zargar-Ahmad]\label{Prop:1.2}
Let $p \in \mathbb{C}[z]$ be of the form:
\begin{equation}\label{Eq:1.2}
p(z)=z^m \prod_{k=1}^{n-m} (z-z_k)
\end{equation}
with $m \geq 1$ integer, and with $|z_k| \geq 1$ for all $k$. Then, $P^{(j)}$ ($1 \leq j \leq m$) has no zero $z$ such that:
$$ 0 < |z| < \prod_{k=0}^{j-1} \frac{m-k}{n-k} = \frac{m(m-1) \cdot \cdot \cdot (m-j+1)}{n(n-1) \cdot \cdot \cdot (n-j+1)}$$
\end{proposition}
In this paper we consider a more general problem: what can we say about other kinds of functions? In order to give an answer to this question, we first need to find a class of functions which satisfies a result similar to the Fundamental Theorem of Algebra. To this aim, we make use of a well known result of Weierstrass, which assures us that entire functions can be written as an infinite product of elementary factors which determine precisely the nonzero roots of $f$, and two other factors (namely, $z^m$ and $e^{g(z)}$ for some positive integer $m$ and some entire function $g$). In the next section, we will recall some important tools from complex analysis which will allow us to prove an extension of the above results.
\section{Preliminaries}
We first recall the following definition:
\begin{definition}\label{Def:2.1}
A complex function $f: \mathbb{C} \rightarrow \mathbb{C}$ of a single variable is called entire if it is holomorphic on the whole complex plane.
\end{definition}
The reason why we consider such functions is the following fundamental result of Weierstrass, which can be found in any textbook on complex analysis (see, for instance, [6,7]):
\begin{theorem}[Weierstrass Factorisation Theorem]\label{Thm:2.1}
Let $f$ be an entire function, and denote by $a_n$ its nonzero roots. Then, there exist a non-negative integer $m$, an entire function $g$ and a sequence of integers $\lbrace p_n \rbrace_{n=1}^{\infty}$ such that:
\begin{equation}\label{Eq:2.1}
f(z)=z^m e^{g(z)} \prod_{n=1}^{\infty} E_{p_n} \left( \frac{z}{a_n} \right) 
\end{equation}
Moreover, $f$ has, as roots, precisely the $a_n$'s and (if $m \geq 1$) $0$.
\end{theorem}
Here, the factors $E_k(z)$ (usually called \textit{elementary factors}) are defined as follows ($k$ non-negative integer):
\begin{equation}\label{Eq:2.2}
    E_k(z):=
    \begin{cases}
      (1-z), & \text{if}\ k=0 \\
      (1-z) \exp \left( \sum_{j=1}^{k} \frac{z^j}{j} \right), & \text{otherwise}
    \end{cases}
\end{equation}
Clearly, polynomials are a special class of entire functions, and indeed, by taking $g(z) \equiv 0$ and $p_n=0$ for all $n$, the expression of $f$ written above reduces to a polynomial (of course, the number of $a_n$ is finite).\\
In addition, there is also another important theorem, which will be used to give some examples of applications of our main result.
\begin{theorem}\label{Thm:2.2}
Let $a_n \neq 0$ (for all $n$) be a sequence such that $|a_n| \rightarrow +\infty$. If $\lbrace p_n \rbrace$ is a sequence of integers such that for all $a>0$:
$$ \sum_{n=1}^{\infty} \left( \frac{a}{|a_n|} \right)^{1+p_n} < +\infty$$
then the function:
$$ f(z):= \prod_{n=1}^{\infty} E_{p_n} \left( \frac{z}{a_n} \right) $$
is entire. In particular, it follows that:
$$ r(z):= z^m e^{g(z)} \prod_{n=1}^{\infty} E_{n} \left( \frac{z}{a_n} \right) $$
is entire for any entire function $g$.
\end{theorem}
We are now ready to state and prove our main result.
\section{Main result}
As for critical points of polynomials (or, more generally, zeros of derivatives of polynomials), the question of determining the distribution and the number of roots of (either real or complex) entire functions has led to many significant results. We refer to [8-13] for more on these topics.\\
In this section, we will prove our main Theorem, which shows that under some suitable conditions, entire functions do not have critical points in punctured balls centred in the origin and of a certain radius, thus extending the result of Aziz and Zargar. In Section 4, instead, we will consider a generalisation of the result of Zargar and Ahmad.
\begin{theorem}\label{Thm:3.1}
Let $f : \mathbb{C} \rightarrow \mathbb{C}$ be a complex entire function, and suppose that it can be written via a Weierstrass factorisation
$$ f(z)=z^m e^{g(z)} \prod_{n=1}^{\infty} E_{p_n} \left( \frac{z}{a_n} \right) $$
with $m$, $g$ and $a_n$ satisfying the following properties:\\
(i) $m \geq 1$ is a positive integer;\\
(ii) $|a_n| \geq h(n) > t$ for some constant $t>0$ and for some function $h: \mathbb{N} \rightarrow \mathbb{R}^{+}$;\\
(iii) $p_n \not \rightarrow 0$ and $\sum_{n=1}^{\infty} \frac{1}{h(n)} < +\infty$ (in particular, $h(n) \rightarrow +\infty$);\\
(iv) for all $z$ such that $0<|z|<t$, the following inequality holds:
$$ |g'(z)| + \sum_{n=1}^{\infty} \frac{1}{h(n) - t} + \sum_{n=1}^{\infty} \frac{1-(t/h(n))^{p_n}}{h(n)-t} < m/t $$
Then, $f'$ has no zeros $z$ in $B(0,t) \setminus \lbrace 0 \rbrace$.
\end{theorem}
\begin{proof}
A simple calculation shows that:
$$ \frac{f'(z)}{f(z)} = \frac{m}{z} + g'(z) + \sum_{n=1}^{\infty} \left( \frac{1}{z-a_n}+ \sum_{k=1}^{p_n} \frac{z^{k-1}}{a_n^{k}} \right) $$
Henceforth, we will only consider $z$ such that $0<|z|<t$. If for some $z$ we have:
$$ \left|g'(z) + \sum_{n=1}^{\infty} \left( \frac{1}{z-a_n}+ \sum_{k=1}^{p_n} \frac{z^{k-1}}{a_n^{k}} \right) \right| > \frac{m}{|z|} $$
then:
$$ \left|\frac{f'(z)}{f(z)} \right| = \left|\frac{m}{z} + g'(z) + \sum_{n=1}^{\infty} \left( \frac{1}{z-a_n}+ \sum_{k=1}^{p_n} \frac{z^{k-1}}{a_n^{k}} \right) \right| \geq \left| \left|g'(z) + \sum_{n=1}^{\infty} \left( \frac{1}{z-a_n}+ \sum_{k=1}^{p_n} \frac{z^{k-1}}{a_n^{k}} \right) \right| - \left| \frac{m}{z} \right| \right| =$$
$$ = \left|g'(z) + \sum_{n=1}^{\infty} \left( \frac{1}{z-a_n}+ \sum_{k=1}^{p_n} \frac{z^{k-1}}{a_n^{k}} \right) \right| - \left| \frac{m}{z} \right| > 0 $$
where the last strict inequality follows by the assumption. Hence, in such cases, $z$ is not a zero of $f'$. Therefore, we can assume that $z$ satisfies:
$$ \left|g'(z) + \sum_{n=1}^{\infty} \left( \frac{1}{z-a_n}+ \sum_{k=1}^{p_n} \frac{z^{k-1}}{a_n^{k}} \right) \right| \leq \left| \frac{m}{z} \right| $$
Now, observe that:
$$ \frac{1}{|z-a_n|} \leq \frac{1}{h(n)-t} $$
since $0<|z|<t$ and $|a_n| \geq h(n) >t$ for all $n$. Hence:
$$ \left|\sum_{n=1}^{\infty} \frac{1}{z-a_n} \right| \leq \sum_{n=1}^{\infty} \frac{1}{|z-a_n|} \leq  \sum_{n=1}^{\infty} \frac{1}{h(n)-t}  $$
where the last series converges by the asymptotic criterion for series ($\frac{1}{h(n)} \sim \frac{1}{h(n)-t}$ when $n \rightarrow +\infty$, and the series with term $\frac{1}{h(n)}$ converges by assumption).\\
Furthermore, we have:
$$ \left| \sum_{n=1}^{\infty} \sum_{k=1}^{p_n} \frac{z^{k-1}}{a_n^{k}} \right| \leq \sum_{n=1}^{\infty} \left| \sum_{k=1}^{p_n} \frac{z^{k-1}}{a_n^{k}} \right| \leq \sum_{n=1}^{\infty} \sum_{k=1}^{p_n} \left|\frac{z^{k-1}}{a_n^{k}} \right| \leq \sum_{n=1}^{\infty} \sum_{k=1}^{p_n} \frac{t^{k-1}}{h(n)^{k}} $$
The finite sum inside the series is a geometric sum, so we get:
$$ \left| \sum_{n=1}^{\infty} \sum_{k=1}^{p_n} \frac{z^{k-1}}{a_n^{k}} \right| \leq \frac{1}{t} \sum_{n=1}^{\infty} \frac{t}{h(n)} \frac{(1-(t/h(n))^{p_n})}{1-t/h(n)} = \sum_{n=1}^{\infty} \frac{(1-(t/h(n))^{p_n})}{h(n)-t} $$
Since $t$ is constant and $h(n) \rightarrow \infty$, $t/h(n) \rightarrow 0$ and therefore (recall that $p_n \not \rightarrow 0$ by hypothesis):
$$ \frac{(1-(t/h(n))^{p_n})}{h(n)-t} \sim \frac{1}{h(n)} $$
Thus, again by the asymptotic criterion for series, we conclude that the last infinite sum converges. Moreover, we can split the first series in its two parts, getting the same as the series at the beginning.\\
Now, by assumption and by what we have just shown:
$$ \left| g'(z) + \sum_{n=1}^{\infty} \frac{1}{z-a_n} +  \sum_{n=1}^{\infty} \sum_{k=1}^{p_n} \frac{z^{k-1}}{a_n^{k}} \right| \leq |g'(z)| + \left| \sum_{n=1}^{\infty} \frac{1}{z-a_n} \right| +  \left| \sum_{n=1}^{\infty} \sum_{k=1}^{p_n} \frac{z^{k-1}}{a_n^{k}} \right| \leq $$
$$ \leq |g'(z)| + \sum_{n=1}^{\infty} \frac{1}{h(n) - t} + \sum_{n=1}^{\infty} \frac{1-(t/h(n))^{p_n}}{h(n)-t} < m/t$$
But $|z|<t$, so that:
$$ \left| g'(z) + \sum_{n=1}^{\infty} \frac{1}{z-a_n} +  \sum_{n=1}^{\infty} \sum_{k=1}^{p_n} \frac{z^{k-1}}{a_n^{k}} \right| < \frac{m}{t} < \left| \frac{m}{z} \right| $$
In particular, it follows that:
$$ \left|\frac{f'(z)}{f(z)} \right| = \left| \frac{m}{z} + g'(z) + \sum_{n=1}^{\infty} \left( \frac{1}{z-a_n}+ \sum_{k=1}^{p_n} \frac{z^{k-1}}{a_n^{k}} \right) \right| \geq $$
$$ \geq \left| \frac{m}{z} \right| - \left|g'(z) + \sum_{n=1}^{\infty} \left( \frac{1}{z-a_n}+ \sum_{k=1}^{p_n} \frac{z^{k-1}}{a_n^{k}} \right) \right| > \frac{m}{t} - \left| g'(z) + \sum_{n=1}^{\infty} \left( \frac{1}{z-a_n}+ \sum_{k=1}^{p_n} \frac{z^{k-1}}{a_n^{k}} \right) \right| >0$$
Thus, $f'$ has no zeros $z$ with $0<|z|<t$.
\end{proof}
We now use Theorem \ref{Thm:2.2} to give a special case of the above result.
\begin{corollary}\label{Crl:3.1}
Let $g$ be any entire function, let $m \geq 1$ be an integer and let $a_n \neq 0$ be a sequence of complex numbers satisfying, for all $n$:
$$ |a_n| \geq h(n) > t $$
for some function $h: \mathbb{N} \rightarrow \mathbb{R}^{+}$ such that $\sum_{n=1}^{\infty} 1/h(n) < +\infty $ and some constant $t>0$. Define:
$$ M:= \max_{z: |z| < t} |g'(z)| $$
and
$$ R:= \frac{2M}{m/t - 2 \sum_{n=1}^{\infty} \frac{1}{h(n) - t}} $$
where
$$ m/t > 2 \sum_{n=1}^{\infty} \frac{1}{h(n) - t} $$
Then, letting
$$ \theta(z):= \frac{g(z)}{R} $$
the entire function
$$ f(z):= z^m e^{\theta(z)} \prod_{n=1}^{\infty} E_n \left( \frac{z}{a_n} \right) $$
has no critical point in $B(0,t) \setminus \lbrace 0 \rbrace$.
\end{corollary}
\begin{proof}
Apply Theorem \ref{Thm:3.1} to the complex function $f$ above. Notice that, by Theorem \ref{Thm:2.2}, $f$ is entire. Clearly, (i) and (ii) are satisfied by hypothesis. Moreover, $p_n=n \not \rightarrow 0$ and hence even (iii) holds true. To conclude, it is easy to see that:
$$ |\theta'(z)| \, + \, \sum_{n=1}^{\infty} \frac{1}{h(n)-t} \, + \, \sum_{n=1}^{\infty} \frac{1-(t/h(n))^{p_n}}{h(n)-t} = |g'(z)/R| \, + \, \sum_{n=1}^{\infty} \frac{1}{h(n)-t} \, + $$
$$+ \, \sum_{n=1}^{\infty} \frac{1-(t/h(n))^{p_n}}{h(n)-t} \leq  M/R \, + \, 2\sum_{n=1}^{\infty} \frac{1}{h(n)-t} = \frac{1}{2} \left( m/t - 2 \sum_{n=1}^{\infty} \frac{1}{h(n) - t}  \right) + $$
$$+ \, 2\sum_{n=1}^{\infty} \frac{1}{h(n)-t} = \frac{m}{2t} \,  + \, \sum_{n=1}^{\infty} \frac{1}{h(n)-t} < m/t $$
where the last inequality follows by the hypothesis. Of course, $\theta$ is entire, so that all the assumptions are satisfied, and hence $f'$ has no critical point in the punctured ball $B(0,t) \setminus \lbrace 0 \rbrace$.
\end{proof}
We conclude this section with some examples.
\begin{example}\label{Ex:3.1}
\normalfont Suppose that $|a_n| \geq n^2 =: h(n)$ (this example can be generalised by replacing the exponent $2$ with any real number $c>1$). Then:
$$ \sum_{n=1}^{\infty} \frac{1}{h(n)} = \zeta(2) = \frac{\pi^2}{6} $$
converges. Consider $m = 13$ and $t=\frac{9}{10}$ (say). It can be shown that:
$$ \sum_{n=1}^{\infty} \frac{1}{h(n) - t} = \frac{5}{9} - \frac{\sqrt{5}}{3 \sqrt{2}} \pi \cot \left( \frac{3 \pi}{\sqrt{10}} \right) \approx 10.737... $$
while (using Mathematica)
$$ \sum_{n=1}^{\infty} \frac{1-\left( \frac{9}{10n^2} \right)^n}{n^2 - \frac{9}{10}} \approx 1.72043...$$
Hence:
$$ 13.568541 < \frac{5}{9} - \frac{\sqrt{5}}{3 \sqrt{2}} \pi \cot \left( \frac{3 \pi}{\sqrt{10}} \right) + \sum_{n=1}^{\infty} \frac{1-\left( \frac{9}{10n^2} \right)^n}{n^2 - \frac{9}{10}} + \frac{10}{9} < 13.569552 $$
Consequently, taking, for instance:
$$ g(z):=z^3-2z^2 $$
and noting that:
$$ M:= \max_{z: |z| < 9/10} |g'(z)| \leq \frac{603}{100}  $$
we can conclude that the function \footnote{Any number $\geq 13.7846$ (in place of $14$) would work as well.}
$$ f(z):= z^{13} e^{z^3/14 -z^2/7} \prod_{n=1}^{\infty} E_n \left( \frac{z}{a_n} \right) $$
has no critical point inside the punctured ball $B(0,9/10) \setminus \lbrace 0 \rbrace$ (this follows from a slight and simple modification of the previous argument). For instance, if $a_n=n^2$ for all $n$, the function:
$$ f_1(z):=z^{13} e^{z^3/14-z^2/7} \prod_{n=1}^{\infty} E_n \left( \frac{z}{n^2} \right) $$
has no critical point in such punctured ball.
\end{example}
\begin{example}\label{Ex:3.2}
\normalfont Suppose that the nonzero roots $a_n$ satisfy:
$$ |a_n| \sim h(n) $$
when $n \rightarrow + \infty$ for some function $h$ such that the seris of its reciprocals converges. Then, by the definition itself of limit, for $n$ large enough (say, $n \geq n_0$) there exists a constant $b>0$ for which $|a_n| \geq h(n) - b >0$. Hence, applying Corollary \ref{Crl:3.1}, for any entire function $g$, the following function has no zeros inside $B(0,t) \setminus \lbrace 0 \rbrace$, with $t>0$ constant such that $|a_n| > 2t$ for all $n$:
$$ f(z):= z^m e^{g(z)/R} \prod_{n=1}^{\infty} E_n \left( \frac{z}{a_n} \right) $$
with
$$ R> \frac{2M}{m/t - 2 \left( \sum_{n=1}^{n_0 -1} \frac{1}{2t-t} + \sum_{n=1}^{\infty} \frac{1}{h(n)-b} \right)}  = \frac{2M}{m/t - 2 \left( \frac{n_0 -1}{t} + \sum_{n=1}^{\infty} \frac{1}{h(n)-b} \right)}  $$
and $M$ defined as the maximum of $|g'(z)|$ over $z: |z|<t$ (we assume that the RHS is $>0$).
\end{example}
By Example \ref{Ex:3.2}, we derive the following corollary of Theorem \ref{Thm:3.1}:
\begin{corollary}\label{Crl:3.2}
Let $g$ be any entire function and let $a_n \neq 0$ be a sequence of complex numbers such that, for $n \rightarrow +\infty$:
$$ |a_n| \sim h(n) $$
for some function $h: \mathbb{N} \rightarrow \mathbb{R}^{+}$ for which $\sum_{n=1}^{\infty} 1/h(n) < +\infty $. Moreover, let $t>0$ satisfy $|a_n| > ct$ ($c>1$) for all $n$. Then, for every positive integer $m \geq 1$ sufficiently large, there exists a constant $L_0$ such that, for every real number $L \geq L_0$, the following entire function has no critical point in $B(0,t) \setminus \lbrace 0 \rbrace$:
$$ f(z):= z^m e^{g(z)/L} \prod_{n=1}^{\infty} E_n \left( \frac{z}{a_n} \right) $$
\end{corollary}
\section{Extension of our main result to higher order derivatives}
Now that we have seen some examples of application of our main result, we are interested in some extension of Theorem \ref{Thm:3.1}, in order to include higher order derivatives, as in [3]. However, the methods used in [3] cannot be applied anymore in our case, so we need to do something else. Here, we will give an algorithm which allows one to obtain a certain inequality which, if satisfied, guarantees the non-existence of critical points in a certain punctured ball. The problem of establishing explicit general assumptions which give a result in the form of Theorem \ref{Thm:3.1} thus remains open, as we will see in the next section.\\
The algorithm begins with the following definitions: we define
$$ A_1(z):=\frac{f'(z)}{f(z)} = \frac{m}{z} + g'(z) + \sum_{n=1}^{\infty} \left( \frac{1}{z-a_n}+ \sum_{k=1}^{p_n} \frac{z^{k-1}}{a_n^{k}} \right) $$
and, more generally:
$$ A_j(z):= \frac{f^{(j)}(z)}{f(z)}$$
(where $f^{(j)}$ is the $j$-th derivative of $f$) In what follows, the terms $A_j(z)$ will be denoted by $A_j$ for simplicity. Notice that:
$$ A_1'= \frac{f''}{f} - \left( \frac{f'}{f} \right)^2 = A_2 - A_1^2 $$
from which we deduce:
$$ A_2=A_1^2+A_1' $$
Proceed in this way:
$$ A_2'= \frac{f'''}{f}-\frac{f''}{f} \frac{f'}{f} = A_3 - A_2 A_1 $$
which implies:
$$ A_3= A_2' + A_2 A_1 $$
and so on. By induction, it can be easily shown that:
$$ A_k = A_{k-1}' + A_{k-1}A_1 $$
Clearly, we can rewrite each $A_k$ only in terms of $A_1$ and its derivatives. After doing this, recalling the explicit form of $A_1$ given in Section 3, we can write:
$$ A_k = L_{m/z} + r $$
where $L_{m/z}$ is the sum of all the terms having a factor of the form $ \frac{m^{\alpha}}{z^{\beta}}$, while $r$ contains all the remaining summands. This split allows us to proceed as follows: we consider, as before, two cases. The first one is the case where some $z$ with $0<|z|<t$ satisfies
$$ |L_{m/z} (z)| < |r(z)| $$
Then:
$$ |A_k| \geq |r(z)| - |L_{m/z}(z)| > 0 $$
so such $z$ is not a zero of $A_k$. Hence, we can only consider the case where $z$ satisfies:
$$ |L_{m/z} (z)| \geq |r(z)| $$
$r$ is a sum of series, derivatives of $g$ and products of such terms. By the triangle inequality, we can find a simpler sum (namely, an upper bound obtained analogously to the one in Section 3) which we can require to be $<$ of a lower bound for $|L_{m/z}|$. In order to do this, notice that we can factor out $\frac{1}{z^{\max \beta}}$ and give a lower bound of the absolute value of this, that is, $1/t^{\max \beta}$. If there is still something depending on $z$ in $|z^{\max \beta} L_{m/z}|$, we have to find a lower bound even for such term. After doing this, the inequality involving the upper and lower bounds that we have found can be seen as an inequality for $|g^{(k)}(z)|$, which, if satisfied, would guarantee that there is no nonzero root of $A_k$ with absolute value $<t$.\\
We show a concrete example of application of this algorithm, which establishes the case for the second derivative of $f$.
\begin{theorem}\label{Thm:4.1}
Let $f: \mathbb{C} \rightarrow \mathbb{C}$ be an entire function and suppose that it can be factorised via Weierstrass Theorem in such a way that:\\
(i) For all $n$, $|a_n| \geq h(n) > t$ for some function $h: \mathbb{N} \rightarrow \mathbb{R}^{+}$ and some constant $t>0$, with $\sum_{n=1}^{\infty} \frac{1}{h(n)} < +\infty$;\\
(ii) $p_n \not \rightarrow 0$ and, when $n \rightarrow +\infty$
$$ p_n \, \frac{t^{p_n -1}}{h(n)^{p_n - 1} \, (h(n)-t)^2} \sim 1/d(n) $$
where $d$ is such that $ \sum_{n=1}^{+\infty} 1/d(n) < +\infty $;\\
(iii) $m$ is a positive integer such that:
$$ m> 1 + 2t \left( \sum_{n=1}^{\infty} \frac{1}{h(n) - t}+ \sum_{n=1}^{\infty} \frac{(1-(t/h(n))^{p_n})}{h(n)-t} \right) $$
(iv) There exists a positive integer $p>1$ for which the following hold for all $z$ such that $|z|<t$:
$$|g'(z)| + \sum_{n=1}^{\infty} \frac{1}{h(n) - t} + \sum_{n=1}^{\infty} \frac{1-(t/h(n))^{p_n}}{h(n)-t} < \frac{m}{pt}$$
and
$$ |g''(z)| < \frac{m}{t^2} \left( m-1 - 2t \left( \sum_{n=1}^{\infty} \frac{1}{h(n) - t}+ \sum_{n=1}^{\infty} \frac{(1-(t/h(n))^{p_n})}{h(n)-t} \right) \right) - \sum_{n=1}^{\infty}  \frac{1}{(h(n)-t)^2}$$
$$ - \, \sum_{n=1}^{\infty} \frac{(p_n -1)(t/h(n))^{p_n+1} -p_n (t/h(n))^{p_n } + t/h(n)}{(t/h(n)-1)^2 \,t \, h(n)} - \frac{m^2}{p^2 t^2} $$
where the RHS is assumed to be $>0$.
\end{theorem}
\begin{proof}
Using the above algorithm, we first notice that:
$$ A_2=A_1'+A_1^2 = \left( -\frac{m}{z^2} + g''(z) - \sum_{n=1}^{\infty}  \frac{1}{(z-a_n)^2} + \sum_{n=1}^{\infty} \sum_{k=1}^{p_n} (k-1) \frac{z^{k-2}}{a_n^{k}} \right) \, + $$
$$+ \, \left( \frac{m}{z} + g'(z) + \sum_{n=1}^{\infty} \frac{1}{z-a_n}+ \sum_{n=1}^{\infty} \sum_{k=1}^{p_n} \frac{z^{k-1}}{a_n^{k}} \right)^2 = \left( -\frac{m}{z^2} + \frac{m^2}{z^2} +2 \frac{m}{z} \left( \sum_{n=1}^{\infty} \frac{1}{z-a_n}+ \sum_{n=1}^{\infty} \sum_{k=1}^{p_n} \frac{z^{k-1}}{a_n^{k}} \right) \right) + $$
$$ + \, \left( g''(z) - \sum_{n=1}^{\infty}  \frac{1}{(z-a_n)^2} + \sum_{n=1}^{\infty} \sum_{k=1}^{p_n} (k-1) \frac{z^{k-2}}{a_n^{k}}  \, + \left( g'(z) + \sum_{n=1}^{\infty} \frac{1}{z-a_n}+ \sum_{n=1}^{\infty} \sum_{k=1}^{p_n} \frac{z^{k-1}}{a_n^{k}} \right)^2  \right) $$
Here we have, using the definitions above:
$$ L_{m/z}(z)= -\frac{m}{z^2} + \frac{m^2}{z^2} +2 \frac{m}{z} \left( \sum_{n=1}^{\infty} \frac{1}{z-a_n}+ \sum_{n=1}^{\infty} \sum_{k=1}^{p_n} \frac{z^{k-1}}{a_n^{k}} \right)$$
and
$$ r(z)= g''(z) \, - \, \sum_{n=1}^{\infty}  \frac{1}{(z-a_n)^2} \, + \, \sum_{n=1}^{\infty} \sum_{k=1}^{p_n} (k-1) \frac{z^{k-2}}{a_n^{k}}  \, + \, \left( g'(z) + \sum_{n=1}^{\infty} \frac{1}{z-a_n}+ \sum_{n=1}^{\infty} \sum_{k=1}^{p_n} \frac{z^{k-1}}{a_n^{k}} \right)^2$$
As noted above, we can suppose that $z$ satisfies:
$$ |L_{m/z}(z)| \geq |r(z)| $$
We now find a lower bound for $|L_{m/z}|$. We factor out $1/z^2$ so that:
$$ |L_{m/z}(z)|= \left| \frac{1}{z^2} \right| \left|-m + m^2 +2mz  \left( \sum_{n=1}^{\infty} \frac{1}{z-a_n}+ \sum_{n=1}^{\infty} \sum_{k=1}^{p_n} \frac{z^{k-1}}{a_n^{k}} \right) \right| \geq $$
$$ \geq \frac{m}{t^2} \left|-1 + m +2z  \left( \sum_{n=1}^{\infty} \frac{1}{z-a_n}+ \sum_{n=1}^{\infty} \sum_{k=1}^{p_n} \frac{z^{k-1}}{a_n^{k}} \right) \right|   $$
By assumption:
$$ \left|2z  \left( \sum_{n=1}^{\infty} \frac{1}{z-a_n}+ \sum_{n=1}^{\infty} \sum_{k=1}^{p_n} \frac{z^{k-1}}{a_n^{k}} \right) \right| <  2t \left( \sum_{n=1}^{\infty} \frac{1}{h(n) - t}+ \sum_{n=1}^{\infty} \frac{(1-(t/h(n))^{p_n})}{h(n)-t} \right) < m-1   $$
so that:
$$ |L_{m/z}(z)| > \frac{m}{t^2} \left( m-1 - 2t \left( \sum_{n=1}^{\infty} \frac{1}{h(n) - t}+ \sum_{n=1}^{\infty} \frac{(1-(t/h(n))^{p_n})}{h(n)-t} \right) \right) > 0 $$
Now that this lower bound has been found, we can proceed with the upper bound for $r$. Note that:
$$ |r(z)| \leq |g''(z)| \, + \, \sum_{n=1}^{\infty}  \frac{1}{(h(n)-t)^2} \, + \, \sum_{n=1}^{\infty} \sum_{k=1}^{p_n} (k-1) \frac{t^{k-2}}{h(n)^{k}}  \, + $$
$$+ \, \left( |g'(z)| + \sum_{n=1}^{\infty} \frac{1}{h(n)-t}+ \sum_{n=1}^{\infty} \frac{(1-(t/h(n))^{p_n})}{h(n)-t} \right)^2 $$
Then, again by hypothesis (notice that the second series converges by (ii)):
$$ |r(z)| \leq |g''(z)| \, + \, \sum_{n=1}^{\infty}  \frac{1}{(h(n)-t)^2} \, + \, \sum_{n=1}^{\infty} \frac{(p_n -1)(t/h(n))^{p_n+1} -p_n (t/h(n))^{p_n } + t/h(n)}{(t/h(n)-1)^2 \, t \, h(n)} + \frac{m^2}{p^2 t^2} $$
Finally, we can conclude that:
$$ \left| \frac{f''(z)}{f(z)} \right| \geq  |L_{m/z}(z)| - |r(z)| > \frac{m}{t^2} \left( m-1 - 2t \left( \sum_{n=1}^{\infty} \frac{1}{h(n) - t}+ \sum_{n=1}^{\infty} \frac{(1-(t/h(n))^{p_n})}{h(n)-t} \right) \right) $$
$$- \left( |g''(z)| \, + \, \sum_{n=1}^{\infty}  \frac{1}{(h(n)-t)^2} \, + \, \sum_{n=1}^{\infty} \frac{(p_n -1)(t/h(n))^{p_n+1} -p_n (t/h(n))^{p_n } + t/h(n)}{(t/h(n)-1)^2 \, t \, h(n)} + \frac{m^2}{p^2 t^2} \right) > 0$$
where the last strict inequality follows by the inequality in the assumptions.
\end{proof}
\section{Some open problems}
The results obtained in the previous sections leave some open problems on the zeros of derivatives of entire functions. In particular, we consider the following ones:
\begin{problem}\label{Prob:5.1}
Is there an optimal constant $t>0$ for which the result of Theorem \ref{Thm:3.1} holds true for all entire functions (with $m$ fixed) satisfying the hypothesis, indipendently of $h(n)$?
\end{problem}
This question arose when considering the special case of polynomials proved by Aziz and Argar. We think that the conditions are too general to have an optimal bound for all such functions, so we expect the answer to this question to be negative. However, we believe that at least the next problem has an affirmative answer:
\begin{problem}\label{Prob:5.2}
For fixed $h$ , $p_n$ and $m \geq 1$, does there exist an optimal value $t>0$ for which any entire $f$ with such a Weierstrass product representation has no critical point in $B(0,t) \setminus \lbrace 0 \rbrace$? If so, find an expression for $t$.
\end{problem}
Shifting the attention to the more general case of higher order derivatives, there is a particular question which needs to be answered in order to completely settle the problem, as we noticed in Section 4.
\begin{problem}\label{Prob:5.3}
Give explicit conditions which assure that the $k$-th derivative of a certain complex entire function $f$ has no roots in $B(0,t) \setminus \lbrace 0 \rbrace$ for some $t>0$. If possible, also find some optimal bounds (fixing $h$, $p_n$ and $m$).
\end{problem}
If solved, this would avoid the use of an algorithm like the previous one, which becomes longer to apply when dealing with higher orders. We remark that, with explicit conditions, we mean assumptions like the ones in Theorem \ref{Thm:4.1}. At the moment, indeed, the unique way to obtain such a result would be by direct evaluation of all the terms involved in the expression of $A_k$, which become too many even with quite low orders of derivation.
\section*{Conclusion}
In this paper we have extended some results of Aziz, Zargar and Ahmad on the zeros of derivatives of polynomials to complex entire functions. We think that the new open problems arising from our results, and in particular Problem \ref{Prob:5.3}, will stimulate further research on these topics.

\end{document}